\documentclass[11pt, a4paper]{article}

\usepackage{lineno}
\usepackage{graphicx}
\usepackage{ae}
\usepackage{amsmath}
\usepackage{amssymb}
\usepackage{latexsym}
\usepackage{url}
\usepackage{epsfig}
\usepackage{cite}
\usepackage{mathrsfs}
\usepackage{amsfonts}
\usepackage{amsthm}

\input amssym.def
\input amssym.tex

\def\De{\Delta}

\def\g{\gamma}

\def\l{\lambda}

\def\ve{\varepsilon}

\def\qed{\hfill$\Box$\vspace{12pt}}

\usepackage{color}

\long\def\delete#1{}

\newcommand{\bmat}[1]{\begin{bmatrix}#1\end{bmatrix}}

\newcommand{\be}{\begin{equation}}
\newcommand{\ee}{\end{equation}}
\newcommand{\ben}{\begin{equation*}}
\newcommand{\een}{\end{equation*}}
\newcommand{\bea}{\begin{eqnarray}}
\newcommand{\eea}{\end{eqnarray}}
\newcommand{\bean}{\begin{eqnarray*}}
\newcommand{\eean}{\end{eqnarray*}}

\def\Cay{{\rm Cay}}

\def\bw{{\rm bw}}

\def\SL{{\rm SL}}

\def\PSL{{\rm PSL}}
\def\PGL{{\rm PGL}}

\def\ZZZ{\mathbb{Z}}

\newtheorem{thm}{Theorem}[section]
\newtheorem{cor}[thm]{Corollary}
\newtheorem{con}[thm]{Construction}
\newtheorem{exam}[thm]{Example}
\newtheorem{lem}[thm]{Lemma}

\newtheorem{rem}{Remark}
\numberwithin{equation}{section}

\title{Spectra of the neighbourhood corona of two graphs}

\author{Xiaogang Liu\; and\; Sanming Zhou
\\
{\small Department of Mathematics and Statistics}\\
{\small The University of Melbourne}\\
{\small Parkville, VIC 3010, Australia}\\
\emph{{\small xiaogliu@student.unimelb.edu.au, smzhou@ms.unimelb.edu.au}} }

\date{}

\begin{document}

\openup 0.5\jot
\maketitle
\begin{abstract}
Given simple graphs $G_1$ and $G_2$, the neighbourhood corona of $G_1$ and $G_2$, denoted  $G_1\star G_2$, is the graph obtained by taking one copy of $G_1$ and $|V(G_1)|$ copies of $G_2$, and joining the neighbours of the $i$th vertex of $G_1$ to every vertex in the $i$th copy of $G_2$. In this paper we determine the adjacency spectrum of $G_1 \star G_2$ for arbitrary $G_1$ and $G_2$, and the Laplacian spectrum and signless Laplacian spectrum of $G_1\star G_2$ for regular $G_1$ and arbitrary $G_2$, in terms of the corresponding spectrum of $G_1$ and $G_2$. The results on the adjacency and signless Laplacian spectra enable us to construct new pairs of adjacency cospectral and signless Laplacian cospectral graphs. As applications of the results on the Laplacian spectra, we give constructions of new families of expander graphs from known ones by using neighbourhood coronae.

\bigskip


\bigskip

\end{abstract}

\section{Introduction}

All graphs considered in this paper are undirected and simple. Let $G=(V(G),E(G))$ be a
graph with vertex set $V(G)=\{v_1,v_2,\ldots,v_n\}$ and edge set
$E(G)$. The \emph{adjacency matrix} of $G$, denoted by $A(G)$, is the $n \times n$ matrix whose $(i,j)$-entry is $1$ if $v_i$ and $v_j$ are adjacent in $G$ and $0$ otherwise. Denote by $d_i=d_G(v_i)$ the degree of $v_i$ in $G$, and define $D(G)$ to be the diagonal matrix with diagonal entries $d_1,d_2,\ldots,d_n$. The \emph{Laplacian matrix} of $G$ and the \emph{signless Laplacian matrix} of $G$ are defined as $L(G)=D(G)-A(G)$ and $Q(G)=D(G)+A(G)$, respectively. Given an $n \times n$ matrix $M$, denote by
$$
\phi(M;x)=\det(xI-M),
$$
or simply $\phi(M)$, the characteristic polynomial of $M$, where $I$ is the identity matrix with the same size as $M$. In particular, for a graph $G$, we call $\phi(A(G))$  (respectively, $\phi(L(G))$, $\phi(Q(G))$) the \emph{adjacency} (respectively, \emph{Laplacian}, \emph{signless Laplacian}) \emph{characteristic polynomial} of $G$, and its roots the \emph{adjacency} (respectively, \emph{Laplacian}, \emph{signless Laplacian}) \emph{eigenvalues} of $G$. Denote the eigenvalues of $A(G), L(G)$ and $Q(G)$, respectively, by
\bean
&&\lambda_1(G)\geq\lambda_2(G)\geq\cdots\geq\lambda_n(G),\\
&&\mu_1(G)\leq\mu_2(G)\leq\cdots\leq\mu_n(G),\\
&&\nu_1(G)\leq\nu_2(G)\leq\cdots\leq\nu_n(G).
\eean
Note that $\mu_1(G) = 0$. The collection of eigenvalues of $A(G)$ (respectively, $L(G)$, $Q(G)$) together with their multiplicities are called the \emph{$A$-spectrum} (respectively, \emph{$L$-spectrum}, \emph{$Q$-spectrum}) of $G$. Two graphs are said to be \emph{$A$-cospectral} (respectively, \emph{$L$-cospectral}, \emph{$Q$-cospectral}) if they have the same $A$-spectrum (respectively, $L$-spectrum, $Q$-spectrum). It is well known that graph spectra store a lot of structural information about a graph; see \cite{kn:Cvetkovic95,kn:Cvetkovic10,kn:Brouwer12} and the references therein.

The \emph{corona} of two graphs was first introduced by Frucht and Harary in \cite{kn:Frucht70} with the goal of constructing a graph whose automorphism group is the wreath product of the two component automorphism groups. Since then a number of papers on graph-theoretic properties of coronae have appeared. As far as eigenvalues are concerned, the $A$-spectra, $L$-spectra and $Q$-spectra of the corona of any two graphs can be expressed by that of the two factor graphs  \cite{kn:Barik07,kn:Cui12,kn:McLeman11,kn:Wang12}. Similarly, the $A$-spectrum, $L$-spectrum and $Q$-spectrum of the edge corona \cite{kn:Hou10} of two graphs, which is a variant of the corona operation, were completely computed in \cite{kn:Cui12,kn:Hou10,kn:Wang12}.

Another variant of the corona operation was introduced in \cite{kn:Gopalapillai11} recently. Given two graphs $G_1$ and $G_2$ on $n_1$ and $n_2$ vertices, respectively, the \emph{neighbourhood corona} of $G_1$ and $G_2$, denoted by $G_1\star G_2$, is the graph obtained by taking one copy of $G_1$ and $n_1$ copies of $G_2$, all vertex-disjoint, and joining every neighbour of the $i$th vertex of $G_1$ to every vertex in the $i$th copy of $G_2$ by a new edge (see Figure \ref{figNB} for an example, where $P_n$ denotes the path with $n$ vertices). In \cite[Theorem 2.1]{kn:Gopalapillai11}, the $A$-spectrum of $G_1\star G_2$ for an arbitrary graph $G_1$ and a regular graph $G_2$, was given in terms of the $A$-spectra of $G_1$ and $G_2$. In the same paper, the author also gave \cite[Theorem 3.1]{kn:Gopalapillai11} the $L$-spectrum of $G_1\star G_2$ in terms of that of $G_1$ and $G_2$, for a regular $G_1$ and an arbitrary $G_2$.

\begin{figure}[here]
\centering
\vspace{-0.9cm}
\includegraphics*[height=7.4cm]{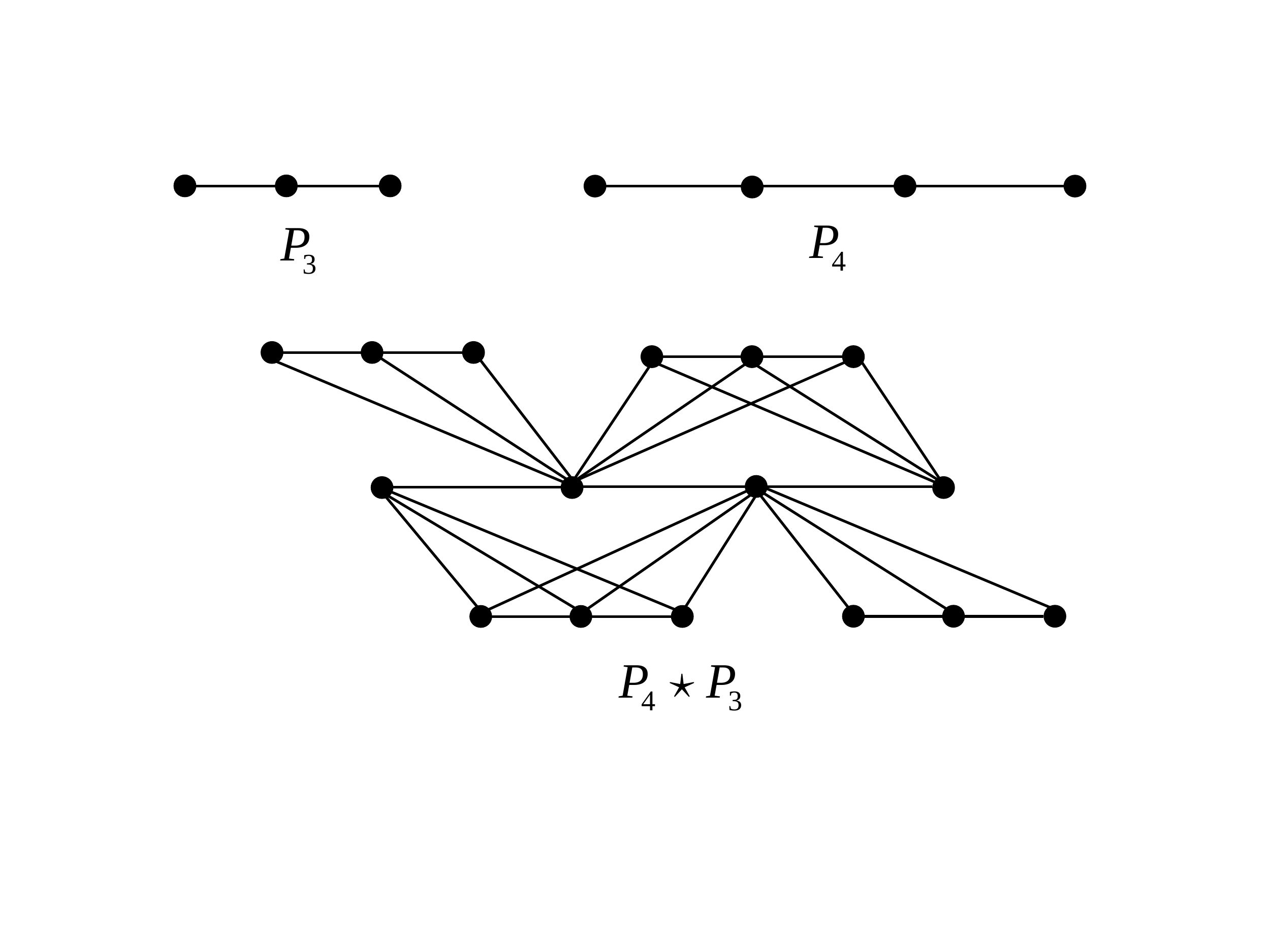}
\vspace{-2.0cm}
\caption{\small The neighbourhood corona of $P_4\star P_3$.}
\label{figNB}
\end{figure}

Motivated by the work above, in the first part of this paper (section \ref{S:spec}), we will first determine the $A$-spectrum of $G_1\star G_2$ for {\em arbitrary} graphs $G_1$ and $G_2$ in terms of that of $G_1$ and $G_2$ (see Theorem \ref{CoronaTh1}); this generalises \cite[Theorem 2.1]{kn:Gopalapillai11}. Second, we will determine the $Q$-spectrum of $G_1\star G_2$ for a regular graph $G_1$ and an arbitrary graph $G_2$ (see Theorem \ref{CoronaTh2}). Thirdly, we notice that the proof of \cite[Theorem 3.1]{kn:Gopalapillai11} about the $L$-spectrum of $G_1 \star G_2$ is incorrect, although the result itself is correct. We will give another proof of this result by using a different approach (see Theorem \ref{CoronaTh3} and Corollary \ref{LCorRegular}). As we will see in Corollaries \ref{cor:acosp} and \ref{cor:qcosp}, our results on $A$-spectra and $Q$-spectra enable us to construct many pairs of $A$-cospectral and $Q$-cospectral graphs, respectively.

Expanders are sparse but highly connected graphs (see section \ref{sec:exp} for a rigorous definition). It is well known \cite{HLW} that expander graphs have many applications in a diversity of disciplines, including computer science, coding theory, cryptography, communication networks, complexity theory, derandomization, Markov chains, statistical mechanics, etc. In most applications there is a need to explicitly construct a family of expander graphs. Because of this, since the first explicit construction \cite{Margulis73}, there has been an extensive body of research (see the survey papers \cite{HLW, Lu}) on explicit construction of expander families from scratch, which is the main theme in this area.

We notice that neighbourhood coronae can be used to construct new families of expander graphs from known ones. As applications of the results on the $L$-spectra of neighbourhood coronae, we will give such constructions in the second part of this paper (section \ref{sec:exp}); see Theorem \ref{thm:exp}, Constructions \ref{con:kr} and \ref{con:cay}, and Examples \ref{ex:pq} and \ref{ex:4}. Though not the main focus of research in the area of expanders, constructing more families of expanders from existing ones helps expand our base of expander families.

\section{Spectra of neighbourhood coronae}\label{S:spec}

In this section, we determine the spectra of neighbourhood coronae with the help of the \emph{coronal} of a matrix. The \emph{$M$-coronal} $\Gamma_M(x)$ of a square matrix $M$ is defined \cite{kn:McLeman11,kn:Cui12} to be the sum of the entries of the matrix $(x I_n-M)^{-1}$, that is,
$$\Gamma_M(x)=\mathbf{1}_n^T(x I_n-M)^{-1}\mathbf{1}_n,$$
where $\mathbf{1}_n$ denotes the column vector of size $n$ with all the entries equal one.

The \emph{Kronecker product} $A\otimes B$ of two matrices $A=(a_{ij})_{m \times n}$ and $B=(b_{ij})_{p \times q}$ is the $mp \times nq$ matrix obtained from $A$ by replacing each element $a_{ij}$ by $a_{ij}B$. This is an associative operation with the property that $(A\otimes B)^T=A^T\otimes B^T$ and $(A\otimes B)(C\otimes D)=AC\otimes BD$ whenever the products $AC$ and $BD$ exist. The latter implies $(A\otimes B)^{-1}=A^{-1}\otimes B^{-1}$ for nonsingular matrices $A$ and $B$. Moreover, if $A$ and $B$ are $n \times n$ and $p \times p$ matrices, then $\det(A\otimes B)=(\det A)^p (\det B)^n$. The reader is referred to \cite{kn:Kronecker} for other properties of the Kronecker product not mentioned here.

Let $G_1$ and $G_2$ be arbitrary graphs on $n_1$ and $n_2$ vertices, respectively. Following \cite{kn:Gopalapillai11}, we first label the vertices of $G_1\star G_2$ as follows. Let $V(G_1)=\left\{v_1,v_2,\ldots,v_{n_1}\right\}$ and $V(G_2)=\left\{u_1,u_2,\ldots,u_{n_2}\right\}$. For $i = 1, 2, \ldots, n_1$, let $u_1^i,u_2^i,\ldots,u_{n_2}^i$ denote the vertices of the $i$th copy of $G_2$, with the understanding that $u_j^i$ is the copy of $u_j$ for each $j$. Denote
$$
W_j=\left\{u_j^1,u_j^2,\ldots,u_j^{n_1}\right\},\;\, j = 1, 2, \ldots, n_2.
$$
Then $V(G_1)\cup W_1\cup W_2\cup\cdots\cup W_{n_2}$ is a partition of $V(G_1\star G_2)$.
It is clear that the degrees of the vertices of $G_1\star G_2$ are:
\be
\label{eq:d}
d_{G_1\star G_2}(v_i) = (n_2+1)d_{G_1}(v_i),\;\, i=1,2,\ldots,n_1
\ee
\be
\label{eq:d1}
d_{G_1\star G_2}(u_j^i) =d_{G_2}(u_j)+d_{G_1}(v_i),\;\, i=1,2,\ldots,n_1,\, j=1,2,\ldots,n_2.
\ee

\subsection{$A$-spectra of neighbourhood coronae}
\label{sec:CoronaA}

\begin{thm}\label{CoronaTh1}
Let $G_1, G_2$ be arbitrary graphs on $n_1, n_2 \ge 1$ vertices, respectively. Then
\begin{eqnarray*}
\phi\left(A(G_1\star G_2);x\right)=\big(\phi(A(G_2);x)\big)^{n_1}\cdot\prod_{i=1}^{n_1}\left(x-\lambda_i(G_1)-\Gamma_{A(G_2)}(x)\lambda_i(G_1)^2\right). \end{eqnarray*}
\end{thm}

\begin{proof}
With respect to the partition $V(G_1)\cup W_1\cup W_2\cup\cdots\cup W_{n_2}$ of $V(G_1\star G_2)$, the adjacency matrix of $G_1\star G_2$ can be written as
\[A(G_1\star G_2)=\bmat{
                          A(G_1) & \mathbf{1}_{n_2}^T\otimes A(G_1) \\[0.2cm]
                          \left(\mathbf{1}_{n_2}^T\otimes A(G_1)\right)^T & A(G_2)\otimes I_{n_1}
                       }.\]
Thus the adjacency characteristic polynomial of $G_1\star G_2$ is given by
\begin{eqnarray*}
\phi\left(A(G_1\star G_2)\right)
&=& \det\bmat{
                          xI_{n_1}-A(G_1) & -\mathbf{1}_{n_2}^T\otimes A(G_1) \\[0.2cm]
                          -\left(\mathbf{1}_{n_2}^T\otimes A(G_1)\right)^T & xI_{n_1n_2}-A(G_2)\otimes I_{n_1}
                        }\\ [0.2cm]
&=& \det\bmat{
                          xI_{n_1}-A(G_1) & -\mathbf{1}_{n_2}^T\otimes A(G_1) \\[0.2cm]
                          -\mathbf{1}_{n_2}\otimes A(G_1) & \left(xI_{n_2}-A(G_2)\right)\otimes I_{n_1}
                       }\\ [0.2cm]
&=& \det\left(\left(xI_{n_2}-A(G_2)\right)\otimes I_{n_1}\right)\cdot\det(S),
\end{eqnarray*}
where
$$
S=xI_{n_1}-A(G_1)-(\mathbf{1}_{n_2}^T\otimes A(G_1))\left(\left(xI_{n_2}-A(G_2)\right)\otimes I_{n_1}\right)^{-1}\left(\mathbf{1}_{n_2}\otimes A(G_1)\right)
$$
is the Schur complement\cite{kn:Schur} of $\left(xI_{n_2}-A(G_2)\right)\otimes I_{n_1}$. The result follows from
\begin{eqnarray*}
\det\left(\left(xI_{n_2}-A(G_2)\right)\otimes I_{n_1}\right)&=&\left(\det(xI_{n_2}-A(G_2))\right)^{n_1}\left(\det (I_{n_1})\right)^{n_2}\\
                                                            &=&\big(\phi\left(A(G_2)\right)\big)^{n_1}
\end{eqnarray*}
and
\begin{eqnarray*}
  \det (S)&=&\det\left(xI_{n_1}-A(G_1)-\left(\mathbf{1}_{n_2}^T\otimes A(G_1)\right)\left(\left(xI_{n_2}-A(G_2)\right)\otimes I_{n_1}\right)^{-1}\left(\mathbf{1}_{n_2}\otimes A(G_1)\right)\right) \\
   &=&  \det\left(xI_{n_1}-A(G_1)-\left(\mathbf{1}_{n_2}^T\big(xI_{n_2}-A(G_2)\big)^{-1}\mathbf{1}_{n_2}\right)A(G_1)^2\right)\\
    &=&  \det\left(xI_{n_1}-A(G_1)-\Gamma_{A(G_2)}(x)A(G_1)^2\right)\\
    &=&\prod_{i=1}^{n_1}\left(x-\lambda_i(G_1)-\Gamma_{A(G_2)}(x)\lambda_i(G_1)^2\right).
\end{eqnarray*}
Here in the last step we used the fact that if $\lambda$ is an eigenvalue of a matrix $A$ and $f(A)$ is a polynomial of $A$, then $f(\lambda)$ is an eigenvalue of $f(A)$.
\qed
\end{proof}

Theorem \ref{CoronaTh1} implies the following result.

\begin{cor}\label{ACorRegular}
\emph{\cite[Theorem 2.1]{kn:Gopalapillai11}}
Let $G_1$ be a graph on $n_1$ vertices and $G_2$ an $r_2$-regular graph on $n_2$ vertices, where $n_1 \ge 1, n_2 \ge 2$ and $r_2 \ge 1$. Then the $A$-spectrum of $G_1\star G_2$ consists of:
\begin{itemize}
  \item[\rm (a)] $\lambda_i(G_2)$, repeated $n_1$ times, for each $i=2,3,\ldots,n_2$;
  \item[\rm (b)] two eigenvalues
\[\frac{\lambda_j(G_1)+r_2\pm\sqrt{\left(\lambda_j(G_1)-r_2\right)^2+4n_2\lambda_j(G_1)}}{2}\]
for each $j=1, 2, \ldots, n_1$.
\end{itemize}
\end{cor}

\begin{proof}
Since $G_2$ is $r_2$-regular with $n_2$ vertices, by \cite{kn:McLeman11,kn:Cui12} we have
$$
\Gamma_{A(G_2)}(x) = \frac{n_2}{x-r_2}.
$$
The only pole of $\Gamma_{A(G_2)}(x)$ is $x=r_2$, which is equal to $\l_1(G_2)$. Thus, by Theorem \ref{CoronaTh1}, for each $i=2,3,\ldots,n_2$, $\lambda_i(G_2)$ is an eigenvalue of $G_1\star G_2$ repeated $n_1$ times. The remaining $2n_1$ eigenvalues of $G_1\star G_2$ are obtained by solving
\[x-\lambda_j(G_1)-\frac{n_2}{x-r_2}\lambda_j(G_1)^2=0\]
for each $j=1,2,\ldots,n_1$, and this yields the eigenvalues in (b).
\qed
\end{proof}

Theorem \ref{CoronaTh1} enables us to compute the $A$-spectrum of many neighbourhood coronae. In general, if we can determine the $A(G_2)$-coronal $\Gamma_{A(G_2)}(x)$, then we are able to compute the $A$-spectrum of $G_1\star G_2$. The following corollary aims to illustrate this method (the case when $p = q$ is also covered by Corollary \ref{ACorRegular}). Denote by $K_{p,q}$ the complete bipartite graph with $p, q \ge 1$ vertices in the two parts of its bipartition.

\begin{cor}\label{ACorComplete}
Let $G$ be a graph on $n \ge 1$ vertices, and let $p, q \ge 1$ be integers. Then the $A$-spectrum of $G \star K_{p,q}$ consists of:
\begin{itemize}
  \item[\rm (a)] $0$, repeated $n \left(p+q-2\right)$ times;
  \item[\rm (b)] the three roots of the equation
 \be
 \label{eq:3}
  x^3-\lambda_j(G)x^2-\left(pq+(p+q)\lambda_j(G)^2\right)x-pq\lambda_j(G)\big(2\lambda_j(G)-1\big)=0
 \ee
 for each $j=1,2,\ldots,n$.
\end{itemize}
\end{cor}

\begin{proof}
It is known \cite[Proposition 8]{kn:McLeman11} that the $A(K_{p,q})$-coronal of $K_{p,q}$ is given by
$$
\Gamma_{A(K_{p,q})}(x) = \frac{(p+q)x+2pq}{x^2-pq}.
$$
The $A$-spectrum of $K_{p,q}$ \cite{kn:Brouwer12,kn:Cvetkovic10} consists of $\pm\sqrt{pq}$ with multiplicity one, and $0$ with multiplicity $p+q-2$. Since $\pm\sqrt{pq}$ are the poles of $\Gamma_{A(K_{p,q})}(x)$, the result follows from Theorem \ref{CoronaTh1} immediately.
\qed
\end{proof}

\begin{rem}
{\em Equation (\ref{eq:3}) may have repeated roots for the same $j$. Also, for different $j$ the corresponding equations (\ref{eq:3}) may have common roots. The multiplicity of each eigenvalue in (b) of Corollary \ref{ACorComplete} is the sum of its multiplicities (which can be $0$) as a root of these equations for $j=1, 2, \ldots, n$.}
\end{rem}

In \cite{kn:Barik07,kn:McLeman11}, many infinite families of pairs of $A$-cospectral graphs are generated by using the corona construction. Similarly, we can use the neighbourhood corona construction to obtain many $A$-cospectral graphs, as stated in the following corollary of Theorem \ref{CoronaTh1}. Note that the condition $\Gamma_{A(G)}(x)=\Gamma_{A(G')}(x)$ in (b) is not redundant because $A$-cospectral graphs may have different $A$-coronals \cite[Remark 3]{kn:McLeman11}.

\begin{cor}
\label{cor:acosp}
Let $G$ and $G'$ be $A$-cospectral graphs, and $H$ an arbitrary graph. Then
\begin{itemize}
  \item[\rm (a)]   $G \star H$ and $G' \star H$ are $A$-cospectral;
  \item[\rm (b)]  $H\star G$ and $H\star G'$ are also $A$-cospectral provided $\Gamma_{A(G)}(x)=\Gamma_{A(G')}(x)$.
\end{itemize}
\end{cor}

\subsection{$Q$-spectra of neighbourhood coronae}
\label{sec:CoronaQ}

\begin{thm}\label{CoronaTh2}
Let $G_1$ be a $r_1$-regular graph on $n_1$ vertices and $G_2$ an arbitrary graph on $n_2$ vertices, where $n_1 \ge 2, n_2 \ge 1$ and $r_1 \ge 1$. Then
\begin{eqnarray*}
\phi\left(Q(G_1\star G_2);x\right)=\big(\phi\left(Q(G_2);x-r_1\right)\big)^{n_1}\cdot\prod_{i=1}^{n_1}\left(x-n_2r_1-\nu_i(G_1)-\Gamma_{Q(G_2)}(x-r_1)(\nu_i(G_1)-r_1)^2\right).
\end{eqnarray*}
\end{thm}

\begin{proof}
Denote by $\mathbf{0}_{n_1\times n_1}$ the all-0 $n_1\times n_1$ matrix.
With respect to the partition $V(G_1)\cup W_1\cup W_2\cup\cdots\cup W_{n_2}$ of $V(G_1\star G_2)$ defined at the beginning of this section, we have
\begin{eqnarray*}
  D(G_1\star G_2)=\bmat{
          (n_2+1)D(G_1) & \mathbf{1}_{n_2}^T\otimes \mathbf{0}_{n_1\times n_1} \\[0.2cm]
          \left(\mathbf{1}_{n_2}^T\otimes \mathbf{0}_{n_1\times n_1}\right)^T & D(G_2)\otimes I_{n_1}+I_{n_2}\otimes D(G_1)
       }
\end{eqnarray*}
and
\[Q(G_1\star G_2)=\bmat{
                          n_2D(G_1)+Q(G_1) & \mathbf{1}_{n_2}^T\otimes A(G_1) \\[0.2cm]
                           \left(\mathbf{1}_{n_2}^T\otimes A(G_1)\right)^T & Q(G_2)\otimes I_{n_1}+I_{n_2}\otimes D(G_1)
                       }.\]
Since $G_1$ is $r_1$-regular, we have
\begin{eqnarray*}
\phi\left(Q(G_1\star G_2)\right)
&=& \det(xI-Q(G_1\star G_2))\\[0.2cm]
&=& \det\bmat{
                          (x-n_2r_1)I_{n_1}-Q(G_1) & -\mathbf{1}_{n_2}^T\otimes A(G_1) \\[0.2cm]
                          -\left(\mathbf{1}_{n_2}^T\otimes A(G_1)\right)^T & xI_{n_1n_2}-Q(G_2)\otimes I_{n_1}-r_1I_{n_2}\otimes I_{n_1}
                        }\\[0.2cm]
&=& \det\bmat{
                          (x-n_2r_1)I_{n_1}-Q(G_1) & -\mathbf{1}_{n_2}^T\otimes A(G_1) \\[0.2cm]
                          -\mathbf{1}_{n_2}\otimes A(G_1) & \big((x-r_1)I_{n_2}-Q(G_2)\big)\otimes I_{n_1}
                        }\\[0.2cm]
&=& \det\left(\big((x-r_1)I_{n_2}-Q(G_2)\big)\otimes I_{n_1}\right)\cdot\det (S),
\end{eqnarray*}
where
$$
S=(x-n_2r_1)I_{n_1}-Q(G_1)-(\mathbf{1}_{n_2}^T\otimes A(G_1))\left(\big((x-r_1)I_{n_2}-Q(G_2)\big)\otimes I_{n_1}\right)^{-1}\left(\mathbf{1}_{n_2}\otimes A(G_1)\right)
$$
is the Schur complement\cite{kn:Schur} of $\big((x-r_1)I_{n_2}-Q(G_2)\big)\otimes I_{n_1}$. Note that $\lambda$ is an eigenvalue of $A(G_1)$ with an eigenvector $\mathrm{v}$
if and only if $\lambda+r_1$ is an eigenvalue of $Q(G_1)$ with the same eigenvector $\mathrm{v}$. Hence
\begin{eqnarray*}
\det\left(\big((x-r_1)I_{n_2}-Q(G_2)\big)\otimes I_{n_1}\right)&=&\left(\det\big((x-r_1)I_{n_2}-Q(G_2)\big)\right)^{n_1}\left(\det I_{n_1}\right)^{n_2}\\
                                                            &=&\big(\phi\left(Q(G_2);x-r_1\right)\big)^{n_1}
\end{eqnarray*}
and
\begin{eqnarray*}
  \det (S) &=&  \det\left((x-n_2r_1)I_{n_1}-Q(G_1)-\left(\mathbf{1}_{n_2}^T\big((x-r_1)I_{n_2}-Q(G_2)\big)^{-1}\mathbf{1}_{n_2}\right)A(G_1)^2\right)\\
    &=&  \det\left((x-n_2r_1)I_{n_1}-Q(G_1)-\Gamma_{Q(G_2)}(x-r_1)A(G_1)^2\right)\\
    &=&\prod_{i=1}^{n_1}\left(x-n_2r_1-\nu_i(G_1)-\Gamma_{Q(G_2)}(x-r_1)(\nu_i(G_1)-r_1)^2\right).
\end{eqnarray*}
The result then follows immediately.
\qed
\end{proof}

Theorem \ref{CoronaTh2} implies that, if we can determine the $Q(G_2)$-coronal, then we are able to compute the signless Laplacian eigenvalues of $G_1\star G_2$. Since this happens when $G_2$ is regular or a complete bipartite graph, we obtain the following two corollaries.

\begin{cor}\label{QCorRegular}
Let $G_1$ be an $r_1$-regular graph on $n_1$ vertices and $G_2$ an $r_2$-regular graph on $n_2$ vertices, where $n_1, n_2 \ge 2$ and $r_1, r_2 \ge 1$. Then the $Q$-spectrum of $G_1\star G_2$ consists of:
\begin{itemize}
  \item[\rm (a)] $r_1+\nu_i(G_2)$, repeated $n_1$ times, for each $i=1,2,\ldots,n_2-1$;
  \item[\rm (b)] two eigenvalues which are the roots of the equation
  \[x^2-\Big((n_2+1)r_1+2r_2+\nu_j(G_1)\Big)x+\Big(2n_2r_1r_2+(2n_2r_1+2r_2+r_1)\nu_j(G_1)-n_2\nu_j(G_1)^2\Big)=0\]
for each $j=1,2,\ldots,n_1$.
\end{itemize}
\end{cor}

\begin{proof}
Since $G_2$ is $r_2$-regular with $n_2$ vertices, we have \cite{kn:Cui12}
$$
\Gamma_{Q(G_2)}(x) = \frac{n_2}{x-2r_2}.
$$
Thus the only pole of $\Gamma_{Q(G_2)}(x-r_1)$ is $x=r_1+2r_2$, and it corresponds to the maximum signless Laplacian eigenvalue $x-r_1=2r_2$ of $G_2$. The result follows from this and Theorem \ref{CoronaTh2}.
\qed
\end{proof}

Similarly, we can compute the $Q$-spectrum of $G_1\star K_{p,q}$. (The case where $p=q$ is also covered by Corollary \ref{QCorRegular}.)

\begin{cor}\label{QCorComplete}
Let $G$ be an $r$-regular graph on $n$ vertices, and let $p,q \ge 1$ be integers. Then the $Q$-spectrum of $G \star K_{p,q}$ consists of:
\begin{itemize}
  \item[\rm (a)]  $p+r$, repeated $n \left(q-1\right)$ times;
  \item[\rm (b)]  $q+r$, repeated $n \left(p-1\right)$ times;
  \item[\rm (c)]  the three roots of the equation
  \be
  \label{eq:3a}
  (x-r)(x-r-p-q)\left(x-(p+q)r-\nu_j(G)\right)-\left((p+q)(x-r)-(p-q)^2\right)(\nu_j(G)-r)^2=0
  \ee
  for each $j=1,2,\ldots,n$.
\end{itemize}
\end{cor}

\begin{proof}
It is known \cite{kn:Cui12} that the $Q(K_{p,q})$-coronal of $K_{p,q}$ is
$$
\Gamma_{Q(K_{p,q})}(x) = \frac{(p+q)x-(p-q)^2}{x^2-(p+q)x}.
$$
Thus the poles of $\Gamma_{Q(K_{p,q})}(x-r)$ are $x=r$ and $x=r+p+q$. It is known \cite{kn:Cui12} that the $Q$-spectrum of $K_{p,q}$ consists of $0$ and $p+q$ each with multiplicity one, $p$ with multiplicity $q-1$, and $q$ with multiplicity $p-1$. From these and Theorem \ref{CoronaTh2} the result follows.
\qed
\end{proof}

\begin{rem}
{\em The multiplicity of each eigenvalue in (c) of Corollary \ref{QCorComplete} is the sum of its multiplicities (which can be $0$) as a root of the equations (\ref{eq:3a}) for $j=1, 2, \ldots, n$.}
\end{rem}

Theorem \ref{CoronaTh2} enables us to construct many pairs of $Q$-cospectral graphs, as shown in the following corollary.

\begin{cor}
\label{cor:qcosp}
\begin{itemize}
  \item[\rm (a)]  If $G_1$ and $G_2$ are $Q$-cospectral regular graphs, and $H$ is any graph, then $G_1\star H$ and $G_2\star H$ are $Q$-cospectral.
  \item[\rm (b)]  If $G$ is a regular graph, and $H_1$ and $H_2$ are $Q$-cospectral graphs with $\Gamma_{Q(H_1)}(x)=\Gamma_{Q(H_2)}(x)$, then $G\star H_1$ and $G\star H_2$ are $Q$-cospectral.
\end{itemize}
\end{cor}
Similar to Corollary \ref{cor:acosp}, the condition $\Gamma_{Q(H_1)}(x)=\Gamma_{Q(H_2)}(x)$ in (b) is not redundant since $Q$-cospectral graphs may have different $Q$-coronals.

\subsection{$L$-spectra of neighbourhood coronae}

The Laplacian spectrum of $G_1 \star G_2$ for a regular graph $G_1$ and an arbitrary graph $G_2$ was given in \cite[Theorem 3.1]{kn:Gopalapillai11}. However, the Laplacian matrix used by the author is incorrect, and this results in incorrect computation throughout the proof, although the final result is correct. In this subsection we fix this problem by giving a different computation of the Laplacian spectrum of $G_1 \star G_2$ using a different approach.

\begin{thm}\label{CoronaTh3}
Let $G_1$ be an $r_1$-regular graph on $n_1$ vertices and $G_2$ an arbitrary graph on $n_2$ vertices, where $n_1 \ge 2, n_2 \ge 1$ and $r_1 \ge 1$. Then
\begin{eqnarray*}
\phi\left(L(G_1\star G_2);x\right)=\big(\phi\left(L(G_2);x-r_1\right)\big)^{n_1}\cdot\prod_{i=1}^{n_1}\left(x-n_2r_1-\mu_i(G_1)-\Gamma_{L(G_2)}(x-r_1)(r_1-\mu_i(G_1))^2\right).
\end{eqnarray*}
\end{thm}

\begin{proof}
Using the notation at the beginning of this section, we have
\[
L(G_1\star G_2) =
\bmat{n_2D(G_1)+L(G_1) & -\mathbf{1}_{n_2}^T\otimes A(G_1) \\[0.2cm]
\left(-\mathbf{1}_{n_2}^T\otimes A(G_1)\right)^T & L(G_2)\otimes I_{n_1}+I_{n_2}\otimes D(G_1)}.
\]
Hence
\begin{eqnarray*}
\phi\left(L(G_1\star G_2)\right)
&=& \det\bmat{(x-n_2r_1)I_{n_1}-L(G_1) & \mathbf{1}_{n_2}^T\otimes A(G_1) \\[0.2cm]
                          -\left(-\mathbf{1}_{n_2}^T\otimes A(G_1)\right)^T & xI_{n_1n_2}-L(G_2)\otimes I_{n_1}-r_1I_{n_2}\otimes I_{n_1}}\\ [0.2cm]
&=& \det\bmat{
                          (x-n_2r_1)I_{n_1}-L(G_1) & \mathbf{1}_{n_2}^T\otimes A(G_1) \\[0.2cm]
                          \mathbf{1}_{n_2}\otimes A(G_1) & \big((x-r_1)I_{n_2}-L(G_2)\big)\otimes I_{n_1}
                        }\\ [0.2cm]
&=& \det\left(\big((x-r_1)I_{n_2}-L(G_2)\big)\otimes I_{n_1}\right)\cdot\det (S),
\end{eqnarray*}
where
$$
S=(x-n_2r_1)I_{n_1}-L(G_1)-(\mathbf{1}_{n_2}^T\otimes A(G_1))\left(\big((x-r_1)I_{n_2}-L(G_2)\big)\otimes I_{n_1}\right)^{-1}\left(\mathbf{1}_{n_2}\otimes A(G_1)\right)
$$
is the Schur complement \cite{kn:Schur} of $\big((x-r_1)I_{n_2}-L(G_2)\big)\otimes I_{n_1}$. Note that $\lambda$ is an eigenvalue of $A(G_1)$ with an eigenvector $\mathrm{v}$
if and only if $r_1-\lambda$ is an eigenvalue of $L(G_1)$ with the same eigenvector $\mathrm{v}$. Hence
\begin{eqnarray*}
\det\left(\big((x-r_1)I_{n_2}-L(G_2)\big)\otimes I_{n_1}\right)&=&\left(\det\big((x-r_1)I_{n_2}-L(G_2)\big)\right)^{n_1}\left(\det I_{n_1}\right)^{n_2}\\
                                                            &=&\big(\phi\left(L(G_2);x-r_1\right)\big)^{n_1}
\end{eqnarray*}
and
\begin{eqnarray*}
  \det (S) &=&  \det\left((x-n_2r_1)I_{n_1}-L(G_1)-\left(\mathbf{1}_{n_2}^T\big((x-r_1)I_{n_2}-L(G_2)\big)^{-1}\mathbf{1}_{n_2}\right)A(G_1)^2\right)\\
    &=&  \det\left((x-n_2r_1)I_{n_1}-L(G_1)-\Gamma_{L(G_2)}(x-r_1)A(G_1)^2\right)\\
    &=&\prod_{i=1}^{n_1}\left(x-n_2r_1-\mu_i(G_1)-\Gamma_{L(G_2)}(x-r_1)(r_1-\mu_i(G_1))^2\right).
\end{eqnarray*}
The result follows directly.
\qed
\end{proof}

It is known \cite[Proposition 2]{kn:Cui12} that, if $M$ is an $n \times n$ matrix with each row sum equal to a constant $t$, then $\Gamma_{M}(x) = n/(x-t)$. In particular, since for any graph $G_2$ with $n_2$ vertices, each row sum of $L(G_2)$ is equal to $0$, we have $\Gamma_{L(G_2)}(x) = n_2/x$. From this and Theorem \ref{CoronaTh3} we obtain the following result.

\begin{cor}\emph{\cite[Theorem 3.1]{kn:Gopalapillai11}}
\label{LCorRegular}
Let $G_1$ be an $r_1$-regular graph on $n_1$ vertices and $G_2$ an arbitrary graph on $n_2$ vertices, where $n_1 \ge 2, n_2 \ge 1$ and $r_1 \ge 1$. Then the $L$-spectrum of $G_1\star G_2$ consists of:
\begin{itemize}
  \item[\rm (a)] $r_1+\mu_i(G_2)$, repeated $n_1$ times, for each $i=2,\ldots,n_2$;
  \item[\rm (b)] two eigenvalues
  \[\frac{1}{2}\left\{(n_2+1)r_1+\mu_j(G_1)\pm\sqrt{\big((n_2+1)r_1+\mu_j(G_1)\big)^2-4\mu_j(G_1)\big((2n_2+1)r_1-n_2\mu_j(G_1)\big)}\right\},\]
for each $j=1,2,\ldots,n_1$.
\end{itemize}
\end{cor}

\section{Application: constructing new expanders from known ones}
\label{sec:exp}

With the help of Corollary \ref{LCorRegular}, in this section we give methods for constructing new families of expander graphs from known ones.

In the literature, $a(G) = \mu_2(G)$ is called the \emph{algebraic connectivity} \cite{kn:Cvetkovic10} of $G$. An infinite family of graphs, $G_{1}, G_2, G_3, \ldots$, is called a family of \emph{$\ve$-expander graphs} \cite{HLW}, where $\ve > 0$ is a fixed constant, if (i) all these graphs are $k$-regular for a fixed integer $k \ge 3$; (ii) $a(G_i) \ge \ve$ for $i = 1, 2, 3, \ldots$; and (iii) $n_i = |V(G_i)| \rightarrow \infty$ as $i \rightarrow \infty$. (Here we use the algebraic connectivity to define expander families. This is equivalent to the usual definition by using the isoperimetric number $i(G)$, because $a(G)/2 \le i(G) \le \sqrt{a(G)(2\De - a(G))}$ \cite{Moh} (see also \cite[Theorem 7.5.15]{kn:Cvetkovic10}) for any graph $G \ne K_1, K_2, K_3$ with maximum degree $\De$.)

Given integers $n, k \ge 1$, let
\be
\label{eq:f}
f_{n,k}(x) =  x + (n+1)k - \sqrt{(4n+1)x^2 - 2(3n+1)k x + ((n +1)k)^2}.
\ee
This is a well-defined real-valued function over $(-\infty, \infty)$.
Define
\be
\label{eq:del}
\delta(n,k,\ve) = \frac{1}{2} \cdot \min\{f_{n,k}(\ve),\, f_{n,k}(2k)\}.
\ee
It can be verified that $\delta(n,k,\ve) > 0$ for $n, k \ge 1$ and $\ve > 0$.

\begin{lem}\label{lem:abound}
Let $G$ be a connected $k$-regular graph with $m$ vertices, and $H \ne K_n$ a connected $r$-regular graph with $n$ vertices. Then
\bea
a(G \star H) & =& \frac{1}{2} \cdot \min\{f_{n,k}(a(G)),\, f_{n,k}(\mu_m(G))\} \label{eq:af1}\\
                   &\ge& \frac{1}{2} \cdot \min\{f_{n,k}(a(G)),\, f_{n,k}(2k)\}. \label{eq:af}
\eea
\end{lem}
\begin{proof}
We have
\bean
f_{n,k}(x) & = & x + (n+1)k - \sqrt{\{(n+1)k+ x\}^2-4x\{(2n+1)k-nx\}} \\
& = & x + (n+1)k - \sqrt{(4n+1) \left\{\left(x-\frac{(3n+1)k}{4n+1}\right)^2 + \frac{4n^3 k^2}{(4n+1)^2}\right\}}
\eean
and
$$
f'_{n,k}(x) = 1 - \left\{(4n+1)x^2 - 2(3n+1)k x + ((n +1)k)^2\right\}^{-1/2} \{(4n+1)x - (3n+1)k\}.
$$
Thus, if $x \le x_0 = (3n+1)k/(4n+1)$, then $f'_{n,k}(x) > 0$ and so $f_{n,k}(x)$ strictly increases with $x$ from $0$ to $x_0$. We now assume $x \ge x_0$. Then $f'_{n,k}(x) > 0$ if and only if
$$
\{(4n+1)x - (3n+1)k\}^2 < (4n+1)x^2 - 2(3n+1)k x + ((n+1)k)^2;
$$
that is,
$$
(4n+1)x^2 - 2(3n+1)kx + (2n+1)k^2 < 0.
$$
Since $x \ge x_0$, one can verify that this inequality holds if and only if $x < k$. In addition, $f'_{n,k}(k) = 0$, and if $k < x < (2n+1)k/n$ then $f'_{n,k}(x) < 0$. Therefore, $f_{n,k}(x)$ strictly increases with $x$ in $[0, k]$ and strictly decreases in $[k, (2n+1)k/n)$.

Since $G$ is $k$-regular, each $\mu_j(G) \le 2k < (2n+1)k/n$.
By Corollary \ref{LCorRegular} and what we just proved it follows that
$$
a(G \star H) =  \min\left\{k + a(H),\, \frac{f_{n,k}(a(G))}{2},\, \frac{f_{n,k}(\mu_m(G))}{2}\right\}.
$$

It can be verified that $f_{n,k}(x) > 0$ when $x \in (0, (2n+1)k/n)$.
Since $G$ is connected and $k$-regular, we have $0 < a(G) \le \mu_m(G) \le 2k$ and hence $f_{n,k}(a(G)), f_{n,k}(\mu_m(G)) > 0$. Since $H \ne K_n$, by \cite[Theorem 7.4.4]{kn:Cvetkovic10} we have $a(H) \le r < n - 1$. It is straightforward to verify that $f_{n,k}(a(G)) \le 2(k + a(H))$ if and only if
$$
\left(na(G) - nk + \frac{a(H)}{2}\right)^2 \ge \left(na(H) - n^2 k + \frac{a(H)}{4}\right) a(H).
$$
Since $a(H) < n - 1$, the latter inequality always holds and therefore we have proved (\ref{eq:af1}).

Since $\sum_{i=1}^m \mu_i(G) = km$ (by considering the trace of $L(G)$), we have $\mu_m(G) \ge k$. Since $f_{n,k}(x)$ strictly decreases in $[k, (2n+1)k/n)$ and $\mu_m(G) \le 2k$, it follows that $f_{n,k}(\mu_m(G)) \ge f_{n,k}(2k) > 0$, and therefore we have (\ref{eq:af}).
\qed
\end{proof}

\begin{thm}
\label{thm:exp}
Suppose $G_1, G_2, G_3, \ldots$ is a family of (non-complete) $k$-regular $\ve$-expander graphs, where $k \ge 3$. Let $H \ne K_n$ be a fixed connected $r$-regular graph on $n$ vertices such that $nk$ is even. Let $G_i(H)$, $i \ge 1$, be any $(n+1)k$-regular graph obtained by adding $n(nk-r)n_i/2$ edges to $G_i \star H$, where $n_i = |V(G_i)|$.
Then $G_1(H), G_2(H), G_3(H), \ldots$ is a family of $\delta(n,k,\ve)$-expander graphs of degree $(n+1)k$.
\end{thm}

In constructing $G_i \star H$, a new edge is allowed to join two non-adjacent vertices in the same copy of $H$. In general, there are many ways to construct $G_i(H)$. Hence different expander families can be constructed even for the same graph $H$ and the same input family $G_1, G_2, G_3, \ldots$. We will give a few specific constructions later.

\bigskip
\begin{proof}
We apply Lemma \ref{lem:abound} to the given family of expanders $G_1, G_2, G_3, \ldots$. Since $G_i \ne K_{n_i}$ for each $i$, where $n_i = |V(G_i)|$, we have $a(G_i) \le k$ by \cite[Theorem 7.4.4]{kn:Cvetkovic10}. Since $f_{n,k}(x)$ strictly increases in $[0, k]$ as we saw in the proof of Lemma \ref{lem:abound}, and since $a(G_i) \ge \ve$, we have $f_{n,k}(a(G_i)) \ge f_{n,k}(\ve)$. Therefore, by (\ref{eq:af}),
\be
\label{eq:ad}
a(G_i \star H) \ge \delta(n,k,\ve).
\ee

Note that the graphs $G_i \star H$ are not regular. In fact, its vertices have degrees $(n+1)k$ or $r+k$, the latter being the degrees of the vertices in copies of $H$. Now let us add $n(nk-r)n_i/2$ edges to $G_i \star H$ such that the resultant graph $G_i(H)$ is $(n+1)k$-regular (and thus no new edge joins a vertex of $G_i$). (We allow a new edge to join two non-adjacent vertices in the same copy of $H$.) This is possible because $n(nk-r)$ is even (as $nr = 2|E(H)|$ and $nk$ is even). Since $G_i \star H$ is a spanning subgraph of $G_i(H)$, and since adding edges to a graph increases or preserves the algebraic connectivity \cite[Theorem 7.1.5]{kn:Cvetkovic10}, we have $a(G_i(H)) \ge a(G_i \star H) \ge \delta(n,k,\ve)$. Therefore, by (\ref{eq:ad}), $G_1(H), G_2(H), G_3(H), \ldots$ is a family of $\delta(n,k,\ve)$-expanders.
\qed
\end{proof}

Theorem \ref{thm:exp} along with the following more explicit construction can be applied recursively to construct new families of expander graphs based on a known family.

\begin{con}
\label{con:kr}
{\em In the case when $k$ and $r$ are both even, one way to construct $G_i(H)$ is as follows. Write $r = sk + t$, so that $t$ is even, where $0 \le t \le k-1$. Choose $U$ and $W$ to be, respectively, $k$-regular and $(k-t)$-regular graphs on the same vertex set $\{u_1, \ldots, u_{n_i}\}$ of size $n_i$. (We allow $U$ and $W$ to have common edges.) In the definition of $G_i \star H$, let $H_j$ denote the copy of $H$ whose vertices are adjacent to the neighbours of the $j$th vertex of $G_i$, $j = 1, 2, \ldots, n_i$. Let us add edges to $G_i \star H$ in the following way. If $u_a$ and $u_b$ are adjacent in $U$, then we add $n(n-s-1)$ edges between $V(H_a)$ and $V(H_b)$ so that they form an $(n-s-1)$-regular bipartite graph $B_{ab}$. If $u_a$ and $u_b$ are adjacent in $W$, then we add a perfect matching $M_{ab}$ between $V(H_a)$ and $V(H_b)$. In the case when $u_a$ and $u_b$ are adjacent in both $U$ and $W$, we add $M_{ab}$ such that it has no common edge with $B_{ab}$. This is possible because the complementary bipartite graph of $B_{ab}$ is an $(s+1)$-regular and hence contains a perfect matching. It can be easily verified that the graph $G_i(H)$ obtained this way is $(n+1)k$-regular. By Theorem \ref{thm:exp}, $G_1(H), G_2(H), G_3(H), \ldots$ is a family of $\delta(n,k,\ve)$-expanders whose structure relies on the choice of $U$ and $W$.

It is noteworthy that we may choose $U$ to be isomorphic to $G_i$. This choice will be especially helpful when dealing with Cayley graphs in our subsequent discussion.}
\qed\end{con}

Given a group $X$ and a subset $S \subseteq X \setminus \{1\}$ such that $S^{-1} = S$, the \emph{Cayley graph} $\Cay(X, S)$ is defined to have vertex set $X$ such that $x, y \in X$ are adjacent if and only if $x^{-1}y \in S$.

\begin{exam}
\label{ex:pq}
{\em Let us apply Construction \ref{con:kr} to the well known Ramanujan graphs $X^{p,q}$ of Lubotzky, Phillips and Sarnak \cite{DSV}, where $p, q$ are distinct odd primes such that $q > 2 \sqrt{p}$.
If $p$ is a square modulo $q$, then $X^{p,q}$ is a non-bipartite $(p+1)$-regular Cayley graph over $\PSL(2, q)$ with $q(q^2-1)/2$ vertices; otherwise, $X^{p,q}$ is a bipartite $(p+1)$-regular Cayley graph over $\PGL(2, q)$ with $q(q^2-1)$ vertices. Here $\PSL(2, q)$ and $\PGL(2, q)$ are, respectively, the projective special linear group of dimension $2$ over the finite field $\mathbb{F}_q$ and the projective general linear group of dimension $2$ over $\mathbb{F}_q$ (see e.g.~\cite{DSV, Dixon-Mortimer}). It is known \cite[Theorem 4.4.4]{DSV} that, for any $\g$ with $0 < \g < 1/6$, $a(X^{p,q}) \ge \ve(\g) = (p+1) - p^{\frac{5}{6} + \g} - p^{\frac{1}{6}-\g}$. Thus, for a fixed odd prime $p$, $\{X^{p,q}: q > 2 \sqrt{p}\;\,\mbox{a prime}\}$ is a family of $(p+1)$-regular $\ve(\g)$-expanders.

Choose a connected $r$-regular graph $H \ne K_n$ with $n$ vertices, where $r \ge 2$ is even. Applying Construction \ref{con:kr} to $H$ and any appropriately chosen $U$ and $W$, we obtain a family of $(n+1)(p+1)$-regular $\delta(n,p+1,\ve(\g))$-expander graphs $\{X^{p,q}(H): q > 2 \sqrt{p}\;\,\mbox{a prime}\}$. Note that $X^{p,q}(H)$ has $(n+1)q(q^2-1)/2$ or $(n+1)q(q^2-1)$ vertices.}
\qed
\end{exam}

The following is another specification of Construction \ref{con:kr} for Cayley expanders.

\begin{con}
\label{con:cay}
{\em Suppose $G_1, G_2, G_3, \ldots, $ is a family of $\ve$-expanders such that each $G_i = \Cay(X_i, S_i)$ is a Cayley graph with degree $k = |S_i| \ge 3$ and order $n_i = |X_i|$. Suppose $H = \Cay(Y, T)$ is a Cayley graph with degree $r = |T| < k$ and order $n = |Y|$ such that $nk$ is even. Then $G_i \star H$ is isomorphic to the graph defined on $X_i \times (Y \cup \{\infty\})$ such that
\begin{itemize}
\item[(a)] $(x, \infty) \sim (x', z)$ if and only if $x^{-1}x' \in S_i$ and $z \in Y \cup \{\infty\}$, and

\item[(b)] for $y \in Y$, $(x, y) \sim (x', y')$ if and only if $x=x', y' \in Y$ and $y^{-1} y' \in T$,
\end{itemize}
where $\sim$ means that the two vertices involved are adjacent.
Choose $U = G_i$ and $W = \Cay(X_i, S_0)$ for some $S_0 \subset S_i$ with $|S_0| = k - r > 0$. Add new edges to $G_i \star H$ in the following way:
\begin{itemize}
\item[(c)] add an edge joining $(x, y)$ and $(x', y')$ if and only if $x^{-1}x' \in S_i$ and $y \ne y'$;

\item[(d)] add an edge joining $(x, y)$ and $(x', y)$ if and only if $x^{-1}x' \in S_0$.
\end{itemize}
Denote the graph obtained this way by $\hat{G}_i$. It can be verified that the adjacency relation of $\hat{G}_i$ is given by (a) and the following rule:
\begin{itemize}
\item[(e)] $(x, y) \sim (x', z)$ if and only if: $x = x'$ and $z \in Y$ such that $y^{-1}z \in T$; or $x^{-1}x' \in S_i$ and $y \ne z \in Y \cup \{\infty\}$; or $x^{-1}x' \in S_0$ and $y = z$.
\end{itemize}
By Theorem \ref{thm:exp}, $\hat{G}_1, \hat{G}_2, \hat{G}_3, \ldots$ is a family of $\delta(n,k,\ve)$-expanders of degree $(n+1)k$.

As a special case, we may choose $H = \Cay(\ZZZ_n, \{1, -1\})$ to be the cycle of length $n$ such that $nk$ is even. If not every element of $S_i$ is an involution, then we can choose $s_i \in S_i$ to be a non-involution and set $S_0 = S_i \setminus \{s_i, s_i^{-1}\}$. In this case, the relation (e) reads as follows: $(x, y) \sim (x, y \pm 1)$ (where the operation on the second coordinate is modulo $n$); $(x, y) \sim (x', z)$ if and only if $x^{-1}x' \in S_i$, and either $y \ne z \in \ZZZ_n \cup \{\infty\}$, or $y=z$ but $x^{-1}x' \ne s_i, s_i^{-1}$.}
\qed
\end{con}

Since many known families of expanders are Cayley graphs \cite{HLW, Lu}, Construction \ref{con:cay} can be used to prove existence and give constructions of many expander families of various degrees based on known Cayley expanders. The following example illustrates this methodology.

\begin{exam}
\label{ex:4}
{\em
Let $a = \bmat{1 & 1\\0 & 1}$ and $b = \bmat{1 & 0\\1 & 1}$, and let $S_m = \{a, a^{-1}, b, b^{-1}\} \subset \SL_2(\ZZZ_m)$.  It is known \cite[Corollary 2.13]{Lu} that the Cayley graphs $G_m = \Cay(\SL_2(\ZZZ_m), S_m)$, for integers $m \ge 2$, form a family of $4$-regular $\ve$-expanders for some $\ve > 0$. Choose $n = 3$ and $S_0 = S_m \setminus \{b, b^{-1}\}$. Define $\hat{G}_m$ to be the graph on $\SL_2(\ZZZ_m) \times (\ZZZ_3 \cup \{\infty\})$ with the following adjacency relations: $(x, j) \sim (x, j \pm 1)$ (for $j \in \ZZZ_3$); $(x, j) \sim (x', \ell)$ if and only if $x^{-1}x' \in S_m$, and either $j \ne \ell \in \ZZZ_3 \cup \{\infty\}$, or $j=\ell$ and $x^{-1}x' = b$ or $b^{-1}$. Then $\{G_m\}_{m \ge 2}$ is a family of $16$-regular $\delta(3, 4, \ve)$-expander graphs, where $\delta(3, 4, \ve) = \ve + 16 - \sqrt{13\ve^2 - 80\ve + 16^2}$ by (\ref{eq:f}) and (\ref{eq:del}).
}
\qed
\end{exam}

Similar to Theorem \ref{thm:exp}, we can construct new expanders from known ones by using edge coronae. Given vertex-disjoint graphs $G_1, G_2$ with $n_1, n_2$ vertices and $m_1, m_2$ edges, respectively, the \emph{edge corona} \cite{kn:Hou10} $G_1 \diamond G_2$ of $G_1$ and $G_2$ is the graph obtained by taking one copy of $G_1$ and $m_1$ copies of $G_2$, and joining two end-vertices of the $i$th edge of $G_1$ to every vertex in the $i$th copy of $G_2$. It is proved in \cite[Theorem 2.4]{kn:Hou10} that, if $G_1$ is $r_1$-regular (where $r_1 \ge 2$) and $G_2$ is $r_2$-regular, then the $L$-spectrum of $G_1\diamond G_2$ consists of: $2$, repeated $m_1-n_1$ times; $2+\mu_i(G_2)$, repeated $m_1$ times, for each $i=2,\ldots,n_2$; and $\frac{1}{2}\left\{r_1n_2+\mu_j(G_1)+2\pm\sqrt{\big(r_1n_2+\mu_j(G_1)+2\big)^2-4(n_2+2)\mu_j(G_1)}\right\}$ for each $j=1,2,\ldots,n_1$.

Define
\ben
g_{n,k}(x) =  x + nk + 2 - \sqrt{(x + nk + 2)^2 - 4(n+2)x}.
\een
Similar to Lemma \ref{lem:abound}, using the above result for $L$-spectra of edge coronae, one can prove that, if $G$ is $k$-regular with $m$ vertices and $H \ne K_n$ is $r$-regular with $n$ vertices, then $a(G \diamond  H) = g_{n,k}(a(G))/2$.
Based on this we obtain the following result whose proof is similar to that of Theorem \ref{thm:exp} and hence is omitted.

\begin{thm}
\label{thm:expE}
Suppose $G_1, G_2, G_3, \ldots$ is a family of (non-complete) $k$-regular $\ve$-expander graphs, where $k \ge 3$. Let $H \ne K_n$ be a fixed connected $r$-regular graph on $n$ vertices. Let $G_i(H)$, $i \ge 1$, be any $(n+1)k$-regular graph obtained by adding $n\big((n+1)k-r-2\big) m_i/2$ edges to $G_i \diamond H$, where $m_i = |E(G_i)|$. Then $G_1(H), G_2(H), G_3(H), \ldots$ is a family of $(g_{n,k}(\ve)/2)$-expander graphs of degree $(n+1)k$.
\end{thm}

\section{Concluding remarks}
\label{sec:rem}

It is well known that the algebraic connectivity of a graph holds the key to understanding several important partition-related invariants such as the isoperimetric number, the bisection width, etc. (By definition, the \emph{bisection width}, denoted by $\bw(G)$, of an $n$-vertex graph $G$ equals $\min\{|\partial S|: S\in V(G),\,|S|=\lfloor n/2 \rfloor\}$, where $\partial S$ is the set of edges of $G$ between $S$ and $V(G) \setminus S$.) Lemma \ref{lem:abound} can be used to bound such invariants for neighbourhood coronae. For example, it is known \cite[Corollary 7.5.3]{kn:Cvetkovic10} that the bisection width $\bw(G)$ of an $n$-vertex graph $G$ is at least $na(G)/4$ if $n$ is even and $(n^2 - 1)a(G)/4n$ if $n$ is odd. From this and (\ref{eq:af}) it follows that, for $G$ and $H$ as in Lemma \ref{lem:abound},
$$
\bw(G\star H)\ge\left\{\begin{array}{ll}
                   \frac{m(n+1)}{8} \cdot \min\{f_{n,k}(a(G)),\, f_{n,k}(2k)\}, & \text{if $m(n+1)$ is even}, \\[0.2cm]
                    \frac{(m(n+1))^2-1}{8m(n+1)} \cdot \min\{f_{n,k}(a(G)),\, f_{n,k}(2k)\}, & \text{if $m(n+1)$ is odd}.
                \end{array}
\right.
$$

Finally, the \emph{vertex expansion} (see e.g. \cite{kn:Cvetkovic10}) of an $n$-vertex graph $G$ is defined as $i^V(G) = \min\{|\delta(S)|/|S|: S \subset V(G), 1 \le |S| \le n/2\}$, where $\delta(S)$ is the set of vertices outside $S$ that are adjacent to some vertices in $S$. It is known \cite[Theorem 7.6.1]{kn:Cvetkovic10} that $i^V(G) \ge 2a(G)/(\De + 2a(G))$ if $G$ has maximum degree $\De$. From this and (\ref{eq:af}) it follows that, for $G$ and $H$ as in Lemma \ref{lem:abound}, we have
$$
i^V(G\star H)\ge \frac{\min\{f_{n,k}(a(G)),\, f_{n,k}(2k)\}}{(n+1)k+\min\{f_{n,k}(a(G)),\, f_{n,k}(2k)\}}.
$$

\bigskip
\noindent \textbf{Acknowledgements}~~X. Liu is supported by MIFRS and MIRS of the University of Melbourne. S. Zhou is supported by a Future Fellowship (FT110100629) of the Australian Research Council.

\end{document}